\documentclass[11pt]{article}

\usepackage{amsmath,amssymb,amsthm,amstext}

\usepackage{graphicx,epsfig}
\usepackage{fullpage}
\date{}
\renewcommand{\le}{\leqslant}
\renewcommand{\ge}{\geqslant}

\newtheorem{theorem}{Theorem}
\newtheorem{corollary}[theorem]{Corollary}
\newtheorem{lemma}{Lemma}
\theoremstyle{remark}
\newtheorem{definition}{Definition}
\newtheorem{remark}{Remark}

\newcommand{\remove}[1]{}

\title{Local rules for the chair tiling from isosceles right triangles}
\author{Nikolay Vereshchagin\\
Moscow State University, HSE University, Yandex%
\thanks{This paper was prepared within the framework of the HSE University Basic Research Program (project no.~HSE-BR-2025-24).}}
\begin{document}

\maketitle
\begin{abstract}
The chair tiling is a classical substitution tiling that cannot be defined by
local rules on its own tiles. We define, by simple local rules, a family of
tilings by square tiles that is mutually locally derivable with a family of
tilings by isosceles right triangles, and show that these rules are perfect: they
admit exactly the substitution tilings for a substitution of the triangles with
inflation factor $\sqrt2$. Forgetting the decoration maps our tilings onto the
chair tilings. Every chair tiling has one, two, four or sixteen preimages, and we
describe exactly when each case occurs. Two steps of our substitution followed by
forgetting the decoration give the same as forgetting the decoration followed by
one step of the chair substitution, while a single step of our substitution
induces no map on the chair tilings. 
%The proof of non-periodicity of our tilings
%is as simple as that of Robinson's tilings.
\end{abstract}  
  \textbf{Keywords:}
non-periodic tilings, Domino problem, aperiodic tile sets, substitution tilings, chair tiling, matching rules
%  \begin{figure}[ht]
%\begin{center}
%\includegraphics[scale=.4]{mp/pic-7.pdf}
%\end{center}
%\caption{}\label{pic7}
%\end{figure}

\section{Introduction}

The study of non-periodic tilings of the plane by square tiles is motivated by the Domino problem: 
to construct an algorithm that, given local rules to attach  tiles, 
finds out whether it is possible to tile the entire plane according to these rules. 
The non-existence of such an algorithm was proved by Berger in~\cite{berger}. 
The main step in the proof is the construction of local rules such that any tiling according to these rules 
is non-periodic. The construction of such a family was subsequently 
simplified by Robinson~\cite{rob}. An even simpler construction using the 
same idea is given in~\cite{dls}.

In this paper, we propose a new such family. 
The local rules for these tilings and the proof of non-periodicity are approximately as easy as 
those in~\cite{rob} and~\cite{dls}. The novel thing is that our tilings are based
on tilings by isosceles right triangles. %; substitution tilings by right
%triangles with inflation factor $\sqrt2$ go back to Danzer and van
%Ophuysen~\cite{lo}, and local rules for related triangle tilings were considered
%by the author in~\cite{ver}.
In fact, we first construct a family of non-periodic tilings with
triangles for which the local rules ensure that the triangles
can be grouped  into square tiles.
Then we rewrite the local rules for triangle tilings as local rules to attach resulting square tiles. 
Similarly, tilings by isosceles triangles with angles divisible by $36^\circ$ yield Penrose tilings P3 by rhombuses, which is called the Robinson decomposition of Penrose tilings. To use conventional terminology,
our square tilings are mutually locally derivable to triangular tilings.

Our square tilings are closely related to the classical chair tiling: forgetting
the colors and the orientations of the sides of the square tiles, we map
every correct tiling onto the arrowed version of a chair tiling, every chair
tiling has at most sixteen preimages, and all preimages are described
(Theorem~\ref{th-chair}). Forgetting the decoration commutes with the
substitutions: two steps of our substitution followed by forgetting the
decoration give the same as forgetting the decoration followed by one step of the
chair substitution. In this sense our substitution is a square root of the chair
substitution; however, one step of our substitution induces no map on the chair
tilings (Remarks~\ref{rem-commute} and~\ref{rem-nodescend}). In particular our family gives local matching rules that
enforce the chair, much as the Trilobite and Crab of Goodman-Strauss~\cite{gscrab}
do through a different decoration; our rules, moreover, are \emph{perfect}, defining
exactly the family of substitution tilings. We discuss these connections in
Section~\ref{s4}.

In the next section, we define our family of tilings with square tiles, and in 
Section~\ref{s3} we define tilings with triangles and provide proofs. 
The proof of the non-periodicity of tilings by triangles (Theorem~\ref{th3}), 
and hence of the non-periodicity of the family of tilings by square tiles (Theorem~\ref{th1}), 
uses a well-known technique based on self-similarity. 
Namely, we consider the following composition scheme: 
two triangles with a common leg are combined into one triangle. 
Applying this scheme to any tiling satisfying local rules we obtain a new tiling that again satisfies the same rules. 
It is well known that any family of tilings with such a property can only consist of non-periodic tilings; the general statement that the unique composition property implies non-periodicity is due to Solomyak~\cite{solomyak}.
Moreover, we prove that all tilings in our family are substitution tilings (Theorem~\ref{th4});
together with the converse this shows that the local rules define \emph{exactly} the substitution
tilings (Corollary~\ref{cor-perfect}). Here a substitution tiling is one each of whose finite
fragments occurs in some supertile, the supertiles being the tilings obtained from single tiles
by several decompositions.

\section{Tilings with square tiles} \label{s2}

Our tiles are shown in Fig.~\ref{pic46}. 
\begin{figure}[ht]
\begin{center}
\includegraphics[scale=2]{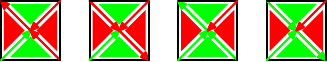}
\end{center}
\caption{The tiles. }
\label{pic46}
\end{figure}
The sides drawn in black have one of two orientations and one of two colors, 
green or red. Thus, each of the four pictures depicts 
$4^4=256$ different tiles. In addition, these tiles can be rotated by angles that are multiples of $45^\circ$.%
\footnote{It may seem that the second tile is obtained from the first by $180^\circ$
rotation, and the fourth from the third. However, this is not the case --- 
with such a rotation, the slanted arrow coming from the upper right corner would change its color.} 
Tiles cannot be flipped. 
Thus, the total number of tiles is $4^4\cdot 4\cdot 8=8192$. A specific tile is shown in Fig.~\ref{pic45}.
\begin{figure}[ht]
\begin{center}
\includegraphics[scale=2]{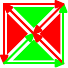}
\end{center}
\caption{A tile. }\label{pic45}
\end{figure}

The local rules:
\begin{enumerate}
\item Tiles are attached only side-to-side, and the shared sides must have matching colors and orientations. 
\item Adjacent triangles 
% (from the pattern on the tiles) 
must have different colors. 
\item We will call  the \emph{crossing} centered at a given vertex the set of arrows entering 
or leaving this vertex.  The rule restricts possible crossings: only the crossings shown
in Fig.~\ref{pic44} are legal, as well as crossings obtained from them by rotations by angles that are multiples of $45^\circ$. 
% (both the arrows from the pattern on the tiles and the arrows on the sides of the tiles are taken into account). 
\begin{figure}[ht]
\begin{center}
\includegraphics[scale=1]{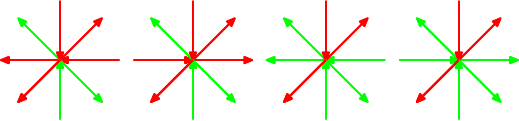}
\end{center}
\caption{Legal crossings. }\label{pic44}
\end{figure}
%These are exactly the same figures that occur in the centers of the tiles, but 
%rotated by $45^\circ$. 
Legal crossings can be described as follows: 
a pair of arrows passes through the center of the crossing without changing color or direction. 
This pair of arrows is called the \emph{axis of the crossing}. 
There are also two arrows of different colors perpendicular to the axis, 
entering the center. In addition to these four arrows, there are four more, coming out of the center at an angle of 
$45^\circ$ to the axis. If you go to the center of the crossing perpendicular to the axis, 
then the green arrow goes to the right, and the red arrow to the left. 
\end{enumerate}
\begin{definition}
A \emph{tiling} is a set of tiles whose interiors are pairwise disjoint.
A  \emph{correct tiling by square tiles} is
a tiling of the whole plane with the above tiles that satisfies these rules.
\end{definition}

Note that by rule 1 all tiles of a tiling of the plane are rotated by the same
angle modulo $90^\circ$: a unit square turned by $45^\circ$ cannot share a side
with a non-turned one. Hence every correct tiling is a tiling of a square grid,
the usual grid of unit squares with vertices in $\mathbb Z^2$ (after a suitable
shift) or this grid turned by $45^\circ$. Rotations by $45^\circ$ are allowed
because the decomposition defined below turns the grid by $45^\circ$.

\begin{theorem}\label{th1}
There is a correct tiling with square tiles, and any such tiling is non-periodic. 
\end{theorem}

The proof of the second part of this theorem uses the \emph{unique composition property}.
The idea is to group the tiles of a correct tiling into so-called macro-tiles, in a way that is unique.
Each macro-tile corresponds to some tile of the original set, 
and when macro-tiles are replaced by the corresponding tiles, the local rules are preserved.

This method will not be applied to the original tilings $S$, 
but to tilings with isosceles right triangles obtained from the original tilings as follows. 
Let us cut each tile from $S$ into 4 right-angled triangles along the diagonal arrows. 
We obtain a tiling with triangles, which is denoted by $\pi(S)$. 
The original local rules become local rules for triangular tiles. 
To these rules, we add yet another rule that guarantees that triangles join into square tiles. 
In any tiling, according to the rules obtained, the tiles are divided into pairs 
of tiles with a common leg. Each such pair will correspond to one triangular tile, 
and when replacing it with the corresponding tile, the local rules will be preserved. 
This operation is called \emph{the composition}. The inverse operation is called \emph{the decomposition}
and is denoted by $\sigma$. Since the local rules are preserved, 
the triangles in the resulting tiling
should again join into square tiles. 
However, new square tiles can be composed of triangles from different square tiles in the original tiling. 
%Therefore, composition can be defined only for tilings with triangles, but not for tilings with square tiles.
Nevertheless, decomposition and composition can be defined also for tilings of the
plane by square tiles. For the tilings $S$ we are interested in (the correct ones,
see Section~\ref{s3}) the triangles of $\sigma\pi(S)$ can again be grouped into
square tiles; the resulting tiling by square tiles is denoted by $\sigma S$, so that
$\pi(\sigma S)=\sigma\pi(S)$. Two remarks are in order. First, $\sigma$ acts only
on tilings of the plane and not on single square tiles: the eight triangles
obtained from the four triangles of one square tile do not form square tiles, so
there is no substitution on square tiles. Second, the square tiles of $\sigma S$
are turned by $45^\circ$ with respect to those of $S$. Two steps of decomposition,
however, can be applied to a single square tile $A$: they produce $16$ triangles
that can be grouped into four square tiles tiling a $2\times2$ square, which we
denote by $\sigma^2A$; the tiling $\sigma^2S$ is the union of the tilings
$\sigma^2A$, $A\in S$, and its square tiles are not turned.
%Here $\pi(S)$ denotes the tiling obtained from $S$ by cutting each square tile into four triangles.

A similar technique is also used in the proof of the non-periodicity of Penrose tilings P3 by rhombuses of two shapes. 
The  rhombuses are also partitioned into triangles, which is called the Robinson decomposition.
Robinson triangles can be grouped into larger triangles,  which again  can be grouped into
rhombuses.

So, we move on to tilings with triangular tiles.

\section{Tilings with triangles}\label{s3}

\subsection{Tiles and local rules}
The tiles are isosceles right triangles colored red or green. 
Each side of a tile has an orientation and a color, 
either red or green. Thus each tile is specified by seven bits. 
In addition, tiles can be rotated by angles that are multiples of $45^\circ$,
so the number of tiles is $2^7\cdot 8=1024$ up to shift.
As in Section~\ref{s2}, rule 1 below implies that all tiles of a tiling of the plane
are rotated by the same angle modulo $90^\circ$ (legs are attached to legs and
hypotenuses to hypotenuses).
Two examples of tiles are shown in Fig.~\ref{pic17}; there the legs of the tiles
are horizontal and vertical, whereas in the tilings $\pi(S)$ defined below (for
$S$ on the usual grid) the hypotenuses are.
\begin{figure}[t]
\begin{center}
\includegraphics[scale=.75]{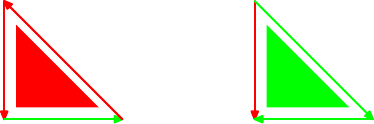}
\end{center}
\caption{Examples of triangular tiles.}
\label{pic17}
\end {figure}

Each correct tiling $S$ with square tiles corresponds to a tiling $\pi(S)$ with triangles, 
which is obtained as follows. We cut each square tile diagonally into four triangles. 
This will result in a tiling with triangular tiles of the type described above. 
Obviously, it satisfies the following local rules:

\begin{enumerate}
\item The tiles are attached side by side, 
and the shared sides of adjacent tiles have matching colors and orientations. 
\item Any two triangles with a common side have different colors.
\item Only two types of crossings are allowed, $C_4$ and $C_8$
 (Fig.~\ref {pic0}).
\begin{figure}[t]
\begin{center}
\includegraphics[scale=1.5]{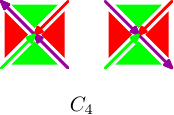}
\hskip .5cm
\includegraphics[scale=1.5]{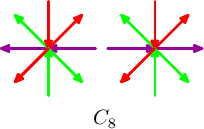}
\end{center}
\caption{Allowed crossings. 
The two sides painted purple must have the same color (red or green). 
Crossings of both forms can be rotated by multiples of $45^\circ$.
}\label{pic0}
\end{figure}
%\begin{quote}
Crossings of the form $C_8$ are exactly legal crossings for tilings with square tiles. 
Crossings of the form $C_4$ appear at the centers of square tiles when they are cut into triangles.
\item
If we go to the center of the $C_4$ crossing perpendicular to the axis, then there should be a red triangle on the left, 
and a green triangle on the right. The axis of the $C_4$ crossing is defined similarly to the axis of $C_8$, 
as the pair of uni-color unidirectional arrows passing through the center.
\end{enumerate}

We will call tilings that satisfy these local rules \emph{correct triangle tilings}. 
For tilings of \emph{parts} of the plane, we require that the third and fourth conditions 
hold only for internal vertices.

\begin{lemma}\label{lem-cut}
(a) If  $S$ is a correct tiling with square tiles, 
then $\pi(S)$ is a correct tiling with triangles. 
(b) Conversely, any correct triangular tiling $T$ is $\pi(S)$ for some correct tiling $S$ with square 
tiles. This tiling $S$ is unique.
\end{lemma}
\begin{proof}
(a) Conditions (1) and (2) for $\pi(S)$ follow from conditions 1 and 2 for $S$. 

Condition (3): every crossing in $\pi(S)$ either coincides with the corresponding crossing in $S$ and hence is of the form $C_8$,
or is the crossing in the center of a tile   in $\pi(S)$ and hence is of the form $C_4$.

Condition (4) is verified by hand.

(b) Since in a correct tiling of the plane by triangles there can only be a crossing $C_4$ at the vertex of any right angle, 
its tiles can be grouped into quadruples, each of which  forms a square tile from our set. Obviously, this partition is unique.

Since the local rules for tilings by squares essentially coincide with the local rules for tilings by triangles, the tiling $S$ is correct.
\end{proof}

 Thus, to prove Theorem~\ref{th1} we have to prove the existence and non-periodicity of correct tilings by triangles.

\subsection{Substitution}

To prove existence and non-periodicity of correct tilings by triangles, 
consider the following substitution $\sigma$: each triangular tile is cut by its height into two tiles similar 
to the original one with the ratio $\sqrt2$. 
The height is oriented from the vertex of the right angle to the hypotenuse and has the same color as the original tile. 
The triangle left from the height, when viewed in the direction of the arrow, is  called  \emph{left}, it is colored red, 
and the right triangle is called \emph{right}, it is colored green (Fig.~\ref{pic5}). 
All sides (including the two halves of the hypotenuse) of the original triangle retain their color and orientation.
\begin{figure}[t]
\begin{center}
\includegraphics[scale=1]{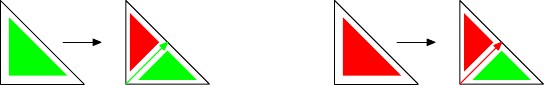}
\end{center}
\caption{Substitution. The sides drawn in black can have any color and orientation, 
which are carried over to the sides of the resulting tiles.}
\label{pic5}
\end{figure}
For example, applying the substitution to the left tile in Fig.~\ref{pic17}, we will get a tiling in Fig.~\ref{pic18}.
\begin{figure}[t]
\begin{center}
\includegraphics[scale=.75]{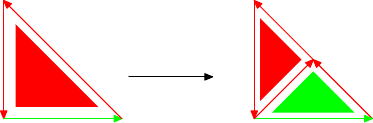}
\end{center}
\caption{An example of applying the substitution}\label{pic18}
\end{figure}

Using this substitution, we define the operations of decomposition and composition of the given tiling.
To decompose a tiling, we apply the substitution to all its tiles.
The resulting tiling tiles the same part of the plane with smaller tiles. 
This tiling is then stretched by a factor of $\sqrt 2$,
and the resulting tiling is called the \emph{decomposition} of the original tiling. 
It is easy to see that the decomposition operation is injective.
Indeed, we call two tiles \emph{siblings} if they are obtained by substitution from the same tile.
Then for each tile $F$ there exists a unique sibling $G$, provided we ignore 
the color and orientation of edges. Indeed, if $F$ is a red tile, then its only sibling is the green tile
obtained from $F$ by $90^\circ$ rotation \emph{counterclockwise}, with respect to the vertex of the right angle of $F$.
If $F$ is a green tile, then the rotation is \emph{clockwise}.

Thus, for each tiling, there is at most one tiling of which it is the decomposition. If such a tiling exists, then it is called \emph{the composition}
of the original tiling. The decomposition of the tiling $T$ will be denoted by $\sigma T$, and the composition by $\sigma^{-1} T$.

\subsection{Supertiles}
\emph{A level $n$} supertile is the $n$-fold decomposition of a single tile. 
Fig.~\ref{pic4} shows a level 6 supertile obtained by $6$-fold decomposition from the left tile in Fig.~\ref{pic18}.
\begin{figure}[h]
\begin{center}
\includegraphics[scale=.8]{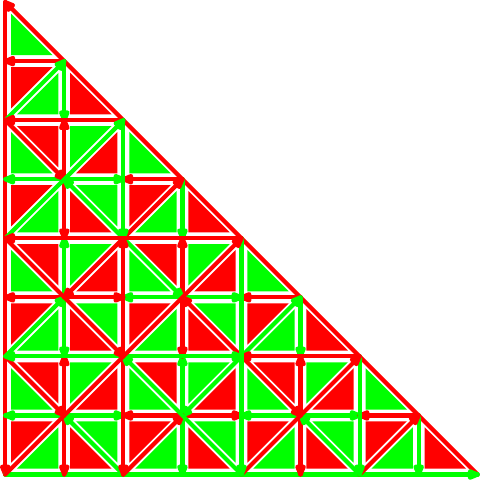}
\end{center}
\caption{Supertile $\Delta_6$. Long arrows represent sequences of equally oriented and equally colored short arrows}\label{pic4}
\end{figure}
Supertiles of level $n$ will be denoted by $\Delta_n$.

\begin{remark}
It can be proven by induction that only three (not 8) different side orientations of inner red tiles and
three different side orientations of inner green tiles occur in supertiles (Fig.~\ref{pic6}). 
Therefore, we can assume that the total number of triangular tiles is $(3+3)\cdot 8\cdot 8=384$ (and not 1024).
\begin{figure}[t]
\begin{center}
\includegraphics[scale=1]{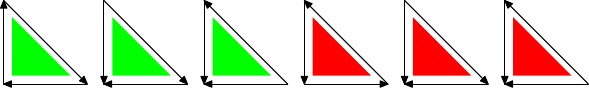}
\end{center}
\caption{Orientations of the sides of tiles occurring in supertiles.}\label{pic6}
\end{figure}
\end{remark}

\subsection{Existence of tilings satisfying local rules}

First we prove the existence of correct tilings by triangles of the entire plane. 
To do this, it suffices to prove that all supertiles satisfy the local rules. 
Why is this enough? For example, we can argue as follows: we can verify that the $\Delta_6$ supertile in 
Fig.~\ref{pic4} contains strictly inside the tile in Fig.~\ref{pic18} on the left, 
from which this supertile was obtained by sixfold decomposition. 
Therefore, if we apply the sixfold decomposition to the supertile in Fig.~\ref{pic4},
we will get a supertile $\Delta_{12}$ including the tiling $\Delta_6$.
Iterating this operation, we get a tower of supertiles 
$$
\Delta_0\subset \Delta_6\subset \Delta_{12}\subset \Delta_{18}\subset\dots, 
$$
whose union is   a correct tiling of the entire plane. 
Thus, it suffices to prove the following
\begin{lemma}\label{lem-super}
Any supertile is a correct tiling. The decomposition of a correct tiling of the
plane is a correct tiling.
\end{lemma}
\begin{proof}
The first statement is proved by induction. The induction base is trivial, since any tile constitutes a correct tiling.

For the inductive step, we have to show that the decomposition of a correct tiling $T$ is again correct. 
But there is a problem here. This is true for tilings of the plane, but not for tilings of its parts. 
For example, a tiling consisting of two triangles with a common leg in Fig.~\ref{pic27},\begin{figure}[ht]
\begin{center}
\includegraphics[scale=.5]{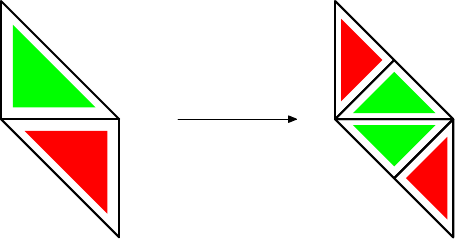}
\end{center}
\caption{The decomposition of a correct tiling of a part of the plane may not be correct.}\label{pic27}
\end{figure}
is correct (the third condition is satisfied because there are no internal vertices), but its decomposition is not. 
To make the inductive step, it is necessary to impose a restriction on the crossings at the boundary points. 
To prove the lemma, it will suffice to require that the crossings with the center at any vertex of the boundary, 
except for the extreme ones, are halves of the legal crossings, that is, they have one of the two forms in Fig.~\ref{pic26}.
\begin{figure}[ht]
\begin{center}
\includegraphics[scale=1]{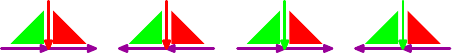}
\vskip .5cm
\includegraphics[scale=1]{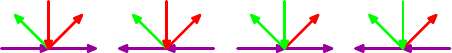}
\end{center}
\caption{Allowed crossings on the boundary of a supertile. The  purple sides must have the same color, red or green.
}\label{pic26}
\end{figure}

Now we can make the inductive step. 
Assume that $T$ is a supertile and is correct. We have to show that so is its decomposition $\sigma T$.
It is obvious that the first condition 
(tiles border side by side and colors and orientations of shared sides match) %for  $\sigma T$
is inherited during decomposition.

Let us verify the second condition (adjacent triangles have different colors) for  $\sigma T$. 
Let $F$ be a tile in $\sigma T$ 
and let $G$ be its sibling; both arise from decomposing a single tile $I$ of $T$. W.l.o.g. assume that  $F$ is a right, and hence green, tile (Fig.~\ref{pic28}).
\begin{figure}[ht]
\begin{center}
\includegraphics[scale=.7]{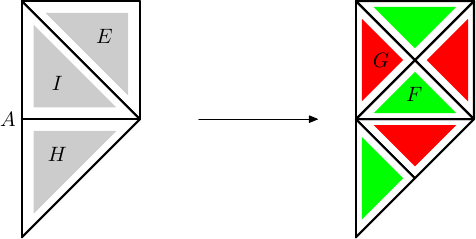}
\end{center}
\caption{All neighbors of a green tile $F$ in the decomposition are red.}\label{pic28}
\end{figure}
We need to prove that all neighbors of $F$ are red.

The tile $G$ is red by definition of the substitution.
Besides $G$, the tile $F$ can have other neighbors.
Namely, if the hypotenuse of $I$ does not lie on the edge of $T$, then the tiling $T$ contains the tile $E$ bordering $I$ along the hypotenuse.
As a result of cutting $E$, two tiles will appear, of which it is the left, and therefore red, tile that borders $F$ along the leg.

Moreover, if the leg of the tile $I$ does not lie on the boundary of the supertile, then the tiling $T$ contains the tile $H$,
which is the reflection of the tile $I$ with respect to this leg.
%The tile $H$ is oriented exactly as it is drawn,
%that is, its right angle adjoins the vertex $A$.
Indeed, otherwise the crossing centered at the vertex $A$ (Fig.~\ref{pic28}) would be illegal
(right and acute angles cannot converge at a legal crossing).
When the tile $H$ is decomposed, it is the left, and therefore red, tile that is adjacent to $F$.

Let us check the third condition for the tiling $\sigma T$ (that all crossings are legal).
To do this, we need to check that during decomposition, a legal crossing is converted into a legal one,
and the same holds for half crossings in Fig.~\ref{pic26}.
This is done by a direct verification.
In addition, we need to check that all the crossings at the new vertices are legal as well.

New vertices are the centers of  hypotenuses of tiles from $T$.
The hypotenuse itself becomes the axis of the new crossing, and the new crossing is $C_4$ by the definition of the substitution.

Finally, we check the fourth condition. The crossings of the form $C_4$ are only the crossings at the centers of the hypotenuses of the tiles in $T$,
and by the definition of substitution, all the left triangles in them are red and the right ones are green.

The third and fourth conditions for crossings centered on the boundary are verified in a similar way.

The same argument, without the boundary considerations, shows that the
decomposition of a correct tiling of the plane is correct.
\end{proof}

\subsection{Non-periodicity of tilings satisfying local rules}

The non-periodicity of correct tilings follows from the fact that any
tiling of the plane has a structure in the following sense:
for any $n$ it can be partitioned uniquely into level $n$ supertiles.
This is a simple consequence of the following 

\begin{theorem}\label{th2}
  Any correct tiling is composable and its composition is correct.
\end{theorem}
\begin{proof}
  Consider any tile of a given correct tiling  $T$ and the vertex of its right angle.
  The crossing at this vertex is $C_4$.
  At this crossing, each tile has a sibling.
  Let us combine them into one, erasing the common leg, and leave all the colors and orientations as they were.
  The hypotenuse of the resulting triangle is formed by connecting two arrows on the axis of the crossing.
  Since they are equally colored and oriented, the hypotenuse will be correctly oriented and correctly colored.
  The resulting tiling is the composition of the original one.
  Let us prove that it is correct.
  
  We are given that $T$ is correct and need to prove that $\sigma^{-1} T$ is correct.
  First we verify condition (1), that is, prove that $\sigma^{-1} T$ is side-to-side and shared sides have matching colors and orientations.
  This is obvious for the legs of the tiles from $\sigma^{-1} T$, since they are the hypotenuses of the tiles from $T$,
  and by the condition (1) for $T$ the hypotenuses are adjacent to the hypotenuses of the same color and orientation.

  Now let us prove the same for the hypotenuses of tiles from $\sigma^{-1} T$.
  Consider a tile $F$ from $\sigma^{-1} T$ (Fig.~\ref{pic13}(a)). W.l.o.g. assume that it is red.
  We need to show that the tile $G$ shown in Fig.~\ref{pic13}(a) belongs to $\sigma^{-1} T$. To this end,
  consider the midpoint of the hypotenuse, denoted by the letter $A$. The crossing centered on $A$ in the tiling $T$ can only be $C_4$,
  so there are two tiles in $T$ under the hypotenuse $A$,
  as shown in Fig.~\ref{pic13}(b). These two tiles, when composed, give the sought  tile $G$.
\begin{figure}[ht]
\begin{center}
\includegraphics[scale=.75]{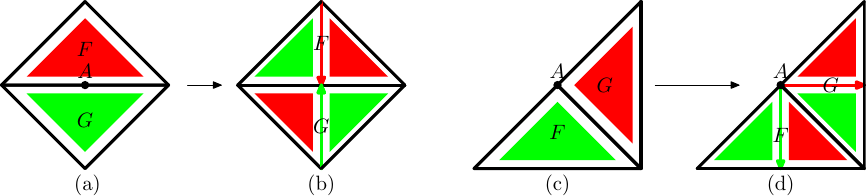}
\end{center}
\caption{Two types of neighborhoods of tiles in the composed tiling. }\label{pic13}
\end{figure} 

As a by-product of this argument, we see that the colors of the tiles $F$ and $G$ are different because the colors of the vertical arrows at the crossing $C_4$ are different.
Therefore, condition (2) is satisfied for triangles $\sigma^{-1} T$ with a common hypotenuse. Let us verify condition (2) for tiles
that share a leg. Let the tiles $F$ and $G$ have a common leg (Fig.~\ref{pic13}(c)).
Then their colors are equal to the colors of two arrows in the tiling $T$ (Fig.~\ref{pic13}(d)),
connecting the point $A$, the common vertex of right angles of $F$ and $G$, with the centers of their hypotenuses.
These two arrows form a right angle and at their common origin (point $A$) there can only be the crossing $C_8$ in the tiling $T$.
By a routine check, we can see that any two orthogonal arrows with a common origin at the center of the crossing $C_8$ have  different colors.

Let us verify condition (3). In the course of composition, the crossings $C_4$ disappear.
We need to show that crossings of the form $C_8$ are transformed into themselves or into $C_4$.
In all crossings $C_8$ the sides of the tiles adjacent to the center of the crossing alternate --- hypotenuse, leg, hypotenuse, and so on.
And the colors of the tiles also alternate. Therefore, 4 options arise, depending on whether the alternation of colors begins with red or green,
and the alternation of sides --- from the leg or hypotenuse.
These four options are shown in Fig.~\ref{pic16}.
 \begin{figure}[ht]
\begin{center}
\includegraphics[scale=.75]{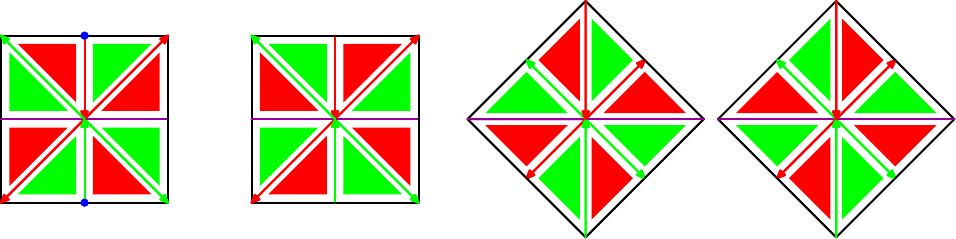}
\end{center}
\caption{Four types of filling $C_8$ crossing. }\label{pic16}
\end{figure}

 Let us first prove that the first option is impossible. Indeed, the crossings at the lower and upper blue points cannot be legal.
 Indeed, the axes of these crossings must be directed vertically, since otherwise an arrow would go from 
 its center in the direction orthogonal to the axis.
 For the upper crossing, the axis should be directed downwards, and for the lower one, upwards.
 But then the colors of both triangles with a vertex at the blue points should be opposite to the existing colors.

 In the next two crossings, the siblings to all tiles are not adjacent to the crossing. For this reason, when composed, the crossings do not change.
 Finally, in the fourth case, the sibling for each tile is also adjacent to the center. When composed, the eight tiles produce four tiles,
 and we get the crossing $C_4$. In this case,
 condition (4) is satisfied for the resulting crossing $C_4$. Indeed, in any $C_8$ crossing,
the red arrow goes to the left of the arrow entering its center, and the green arrow goes to the right.
Hence a red triangle is to the left of the arrow entering the center of the obtained $C_4$ crossing,
and a green one to the right.
\end{proof}

\begin{theorem}\label{th3}
Any correct tiling of the plane by triangles is non-periodic. Consequently, any correct tiling by square tiles is also non-periodic.
\end{theorem}
\begin{proof}
  Let $T$ be a correct tiling of the plane by triangles and $\sigma^{-1}T$ be its composition.
Assume that $a\ne 0$ is its period, that is $T=T+a$.
Then the vector $a/\sqrt2$ is the period of $\sigma^{-1}T$. Indeed,
$$
\sigma^{-1}T=
\sigma^{-1}(T+a)= (\sigma^{-1}T)+a/\sqrt2.
$$
It is proved similarly that $a/2$ is the period of the double composition of $T$.
Repeating this argument many times, we get a tiling whose period is much less than the size of the tiles, which is impossible.
\end{proof}

\subsection{Substitution tilings}

\begin{definition}
  A tiling is called a \emph{substitution tiling} (for the given substitution) if every its finite fragment
  occurs in a supertile.
\end{definition}

Since supertiles are correct tilings, any substitution tiling is correct. It turns out that the opposite is also true.

\begin{theorem}\label{th4}
Any correct tiling of the plane by triangles is a substitution tiling.
\end{theorem}
\begin{proof}
Consider any fragment  $P$
of a correct tiling $T$ of the  plane. We have to show that $P$ occurs in a supertile. To this end add in $P$ a finite number of tiles 
from  $T$ so that $P$ becomes an inner part of the resulting fragment $Q$ of $T$.

By Theorem~\ref{th2} for any $k$ we can compose the given tiling  $k$ times
and the resulting tiling $\sigma^{-k}T$ is correct.  
Call a \emph{crown in $\sigma^{-k}T$ centered at a vertex $A$} of a tile in $\sigma^{-k}T$
the set of all tiles from $\sigma^{-k}T$ that include $A$.  Since  $\sigma^{-k}T$
is a correct tiling, all its crowns have a form shown on Fig.~\ref{pic31} (see the proof of Th.~\ref{th2}).
 \begin{figure}[ht]
\begin{center}
\includegraphics[scale=1]{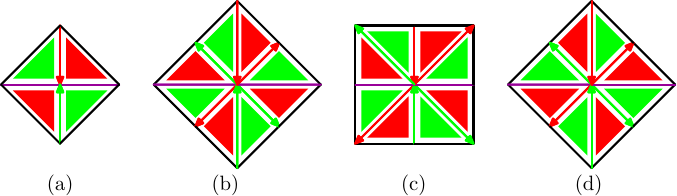}
\end{center}
\caption{Crowns in a correct tiling. Black and purple sides can have any colors and orientations. The purple 
sides must have matching colors and orientations. }\label{pic31}
\end{figure}

Consider the sets of the form $\sigma^k C$ where $C$ is a crown within the tiling
$\sigma^{-k}T$. 
As $k$ increases, these sets increase as well. If $k$ is large enough, then the set $Q$ 
is covered by a  single such set, say by $\sigma^k C$,
that is,  $Q\subset \sigma^k C$. As $\sigma^{-k}T$  
is a correct tiling, it can have only crowns shown on   Fig.~\ref{pic31}.

Let us show that all crowns from 
 Fig.~\ref{pic31} appear in  supertiles provided we ignore colors and orientations of its outer 
 sides (shown in black color on Fig.~\ref{pic31}).  For crowns of the form (a) 
it can be verified  by hand: all the four such crowns occur within  supertiles $\Delta_5$ shown on Fig.~\ref{pic34}.
 \begin{figure}[ht]
\begin{center}
\includegraphics[scale=.75]{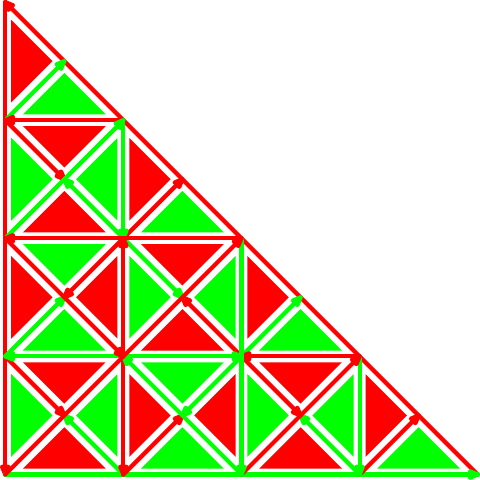}\quad \includegraphics[scale=.75]{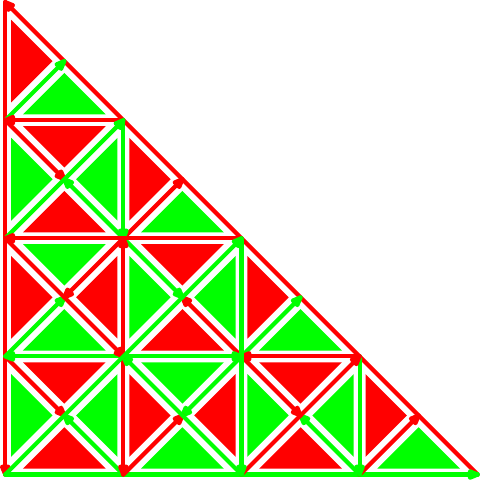}
\end{center}
\caption{On diagonals of supertiles $\Delta_5$, there are all the four crowns of the form (a). }\label{pic34}
\end{figure}
Again it is easy to verify  by hand that the substitution transforms the crowns as follows:
(a)$\to$(b)$\to$(c)$\to$(d)$\to$(c). Hence all  the four crowns (b) occur within supertiles $\Delta_6$,
the crowns (c)  within  supertiles $\Delta_7$,
and the crowns (d)  within supertiles $\Delta_8$.

It follows that the crown $C$, whose $k$-fold decomposition covers $Q$, 
appears in a supertile, say in $\Delta_n$. Therefore
the tiling $\sigma^{k}C$ appears in the supertile $\Delta_{n+k}$. Hence the fragment $Q$  
appears in that supertile provided we ignore colors and orientations of its
outer sides. Since $P\subset Q$ and no side of the  fragment $P$ is an outer side of $Q$,
we are done.
\end{proof}

Combining Theorem~\ref{th4} with the remark preceding it (supertiles are
correct tilings), we obtain an exact description of the family of tilings
defined by our local rules.

\begin{corollary}\label{cor-perfect}
A tiling of the plane by triangles is correct if and only if it is a
substitution tiling; that is, the local rules of Section~\ref{s3} admit exactly
those tilings whose every finite fragment occurs in a supertile.
\end{corollary}
\begin{proof}
If a tiling is correct, then it is a substitution tiling by Theorem~\ref{th4}.
Conversely, correctness is a local property: each of the conditions (1)--(4)
constrains only the tiles meeting at a single vertex. If every finite fragment
of a tiling occurs in a supertile, then in particular the fragment formed by
the tiles around any given vertex occurs in a supertile; enlarging it if
necessary, we may assume that it occurs in the interior of a supertile, where
all four conditions hold because supertiles are correct tilings. Hence the
tiling is correct.
\end{proof}

\begin{remark}\label{rem-perfect}
Corollary~\ref{cor-perfect} says that our local rules are \emph{perfect}: the
family they define is precisely the family of substitution tilings, with no
extra tilings admitted. This is stronger than merely forcing the substitution
\emph{hierarchy}. Every substitution tiling is hierarchical (admits an infinite
chain of compositions), but the converse may fail~\cite{gsrev,gs26}; and matching
rules that force only the hierarchy typically admit additional non-substitution
tilings, for instance non-repetitive tilings with fault lines~\cite{frank}. In
the language of~\cite{gsrev}, Corollary~\ref{cor-perfect} exhibits the
$\sqrt2$-triangle substitution as a sofic family, witnessed by an explicit and
simple set of perfect matching rules.
\end{remark}

The following observation will be used repeatedly.

\begin{lemma}\label{lem-inside}
Let $t$ be a triangular tile. The colors of all triangles of $\sigma^m t$ and the
colors and orientations of all edges of $\sigma^m t$ that do not lie on the
boundary of $\sigma^m t$ do not depend on the colors and orientations of the
sides of $t$; they depend only on the color of $t$.
\end{lemma}
\begin{proof}
By induction on $m$. The substitution colors the two children of $t$ as left and
right, that is, by geometry alone, and it gives the new edge (the height) the
color of $t$ and the orientation from the vertex of the right angle. The
interior edges of $\sigma^m t$ are this height (cut into pieces that keep its
color and orientation) and the interior edges of the $(m-1)$-fold decompositions
of the two children, to which the induction hypothesis applies.
\end{proof}

Corollary~\ref{cor-perfect} has an analogue for tilings by square tiles. For a
square tile $\Box$ and $n\ge0$ the tiling $\sigma^{2n}\Box$ of a $2^n\times2^n$
square by square tiles was defined in Section~\ref{s2}.

\begin{definition}\label{def-subst-sq}
A tiling $S$ of the plane by square tiles is a \emph{substitution tiling} if every
finite fragment of $S$ lies in a tiling $\sigma^{2n}\Box$ for some $n\ge0$ and
some square tile $\Box$.
\end{definition}

\begin{corollary}\label{cor-perfect-sq}
Let $S$ be a tiling of the plane by square tiles. Then $S$ is a substitution
tiling if and only if $\pi(S)$ is a substitution tiling. Consequently, $S$ is
correct if and only if it is a substitution tiling.
\end{corollary}
\begin{proof}
The second statement follows from the first one by Lemma~\ref{lem-cut} and
Corollary~\ref{cor-perfect}. 

To prove the first statement, consider the following two claims:
\begin{itemize}
\item For every square supertile $\sigma^{2n}\Box$ the tiling $\pi(\sigma^{2n}\Box)$ is included in some triangle supertile.
\item Every triangle supertile $\sigma^mt$ is included in $\pi(\sigma^{2n}\Box)$ for some square supertile $\sigma^{2n}\Box$.
\end{itemize}
The first statement would follow from these claims immediately, and it would
suffice to prove the claims for level 0 (single tiles). Unfortunately, the claims
for level 0 are false: for instance, not all triangle tiles occur
in supertiles, as we mentioned earlier.

However, we can deduce the statement from relaxed claims, using the following trick.
Say that a tiling $F$ is \emph{almost included} in a tiling $F'$ if $F$ is included
in $F'$ provided we ignore the colors and orientations of the boundary edges of
$F$ --- they may differ from those of the corresponding edges of $F'$. Let $F$ be a
finite fragment of a tiling of the plane (by square tiles or by triangles) and let
$F'$ be the fragment of that tiling formed by $F$ and all tiles meeting $F$. Every
edge of $F$ is an interior edge of $F'$; hence if $F'$ is almost included in a
tiling $F''$, then $F$ is included in $F''$ in the usual sense. Therefore the first
statement follows from the following relaxed claims:
\begin{itemize}
\item For every square supertile $\sigma^{2n}\Box$ the tiling $\pi(\sigma^{2n}\Box)$ is \emph{almost} included in some triangle supertile.
\item Every triangle supertile $\sigma^mt$ is \emph{almost} included in $\pi(\sigma^{2n}\Box)$ for some square supertile $\sigma^{2n}\Box$.
\end{itemize}
Almost inclusion commutes with decomposition: if $F$ is almost included in $F'$,
then $\sigma F$ is almost included in $\sigma F'$, since the substitution reads only
the color of a tile and copies its sides to the boundary of its children. Hence
it suffices to prove the relaxed claims for tiles: $\pi(\Box)$ is almost included in
a triangle supertile for every square tile $\Box$, and every triangle tile $t$ is
almost included in $\pi(\Box)$ for some square tile $\Box$. 

For the second claim this reduction works only for even $m$: from $t$ almost included in $\pi(\Box)$ we get
$\sigma^mt$ almost included in $\sigma^m\pi(\Box)$, and $\sigma^m\pi(\Box)=\pi(\sigma^m\Box)$
for the square supertile $\sigma^m\Box$ if $m$ is even. However, for odd $m$ the tiling $\sigma^m\Box$ is not defined, as square supertiles
exist only for even levels. For odd $m$ we use in addition that $t$ is almost
included in $\sigma t'$ for every triangle tile $t'$ (one of the two children of
$t'$ has the color of $t$). Hence $\sigma^mt$ is almost included in
$\sigma^{m+1}t'\subset\pi(\sigma^{m+1}\Box)$ whenever $t'$ is a triangle of
$\pi(\Box)$.

Both relaxed claims for tiles are verified by inspection. When we ignore the colors and
orientations of sides, square tiles have four different types up to rotation
(the four pictures in Fig.~\ref{pic46}), and the supertile $\Delta_6$ of
Fig.~\ref{pic4} contains square tiles of all four types up to rotation; as the
sets of tiles are closed under rotations by multiples of $45^\circ$, so are the
supertiles, and the first claim follows. Likewise, triangle tiles have two types,
red and green, and $\pi(\Box)$ contains triangles of both colors for any square
tile $\Box$; this gives the second claim.
\end{proof}

\section{Relation to the chair tiling}\label{s4}

In this section we explain how our tilings are related to the classical
\emph{chair tiling} and to other local rules that enforce it. The chair tiling
is the substitution tiling generated by the L-tromino (a $2\times2$ square with
one of its four unit squares removed) taken in four orientations; the chair
substitution inflates by a factor of~$2$ and replaces each chair by four chairs.
The chair substitution admits no local matching rules on its own tiles: no
finite set of rules on the undecorated L-trominoes has the chair tilings as
exactly its tilings, so the chair can be pinned down only through a decoration
--- as our construction below, and the Trilobite and Crab, do (see~\cite{gschair}
for background, \cite{bg} for the general theory, and the Tilings
Encyclopedia~\cite{te} for a catalogue of the chair and related substitutions).

The chair tiling is mutually locally derivable to the \emph{arrowed-square
tiling} (the \emph{square chair} of~\cite{bms,te}): the plane is tiled by a
grid of unit squares (as for our square tiles, the grid may be turned by
$45^\circ$), and each square carries a diagonal arrow pointing to one
of its four corners. We denote the four arrows (and the corresponding directions)
by the vectors
\[
c\in\{(1,1),(-1,1),(-1,-1),(1,-1)\}:
\]
the arrow $c$ points from the center of the square to the corner lying in the
direction $c$, and $-c$ is the opposite direction. Accordingly, for a square
divided into four quarters we speak of the quarter (or the corner) in the
direction $c$. In all figures of this section we draw this arrow as a
small square of a different color, placed in the corner the arrow points to. The
four arrow directions correspond to the four orientations of the chair, and the
chair substitution induces the \emph{arrowed-square substitution} $\tau$ on the
arrows (Figs.~\ref{pic-chairsub}--\ref{pic-arrowsub}). We freely identify a chair tiling
with its arrowed-square form; concretely, in the theorem below a \emph{chair
tiling} is an arrowed-square tiling each of whose finite fragments occurs in an
arrowed-square supertile, which by the bijection of Remark~\ref{rem-mld} is the
same as an L-tromino chair tiling rendered on arrowed squares.

\begin{figure}[ht]
\begin{center}
\includegraphics[scale=1.2]{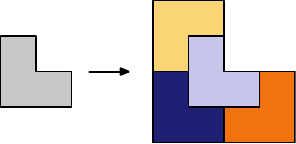}
\end{center}
\caption{The chair substitution: an L-tromino is inflated by~$2$ and cut into
four chairs (shown in four colors).}
\label{pic-chairsub}
\end{figure}

\begin{figure}[ht]
\begin{center}
\includegraphics[scale=1.1]{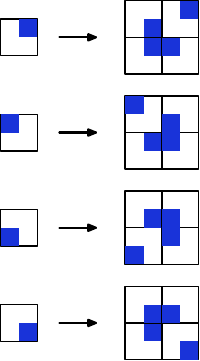}
\end{center}
\caption{The arrowed-square substitution, mutually locally
derivable to the chair. Each unit square carries a marker --- the small square
in the corner its diagonal arrow points to --- in one of its four corners. Under
the substitution the child lying in the direction of the parent's marker keeps
that marker (pointing outward), while the other three point to the center of the
$2\times2$ square.}
\label{pic-arrowsub}
\end{figure}

\begin{remark}\label{rem-mld}
The chair substitution and the arrowed-square substitution of
Figs.~\ref{pic-chairsub}--\ref{pic-arrowsub} are mutually locally derivable: on
their substitution tilings there is a local bijection taking one to the other.
In one direction, complete every chair by its marker (the missing unit square of
its $2\times2$ block) to a square tile of side $2$, and delete the tiles that are
entirely covered by other tiles; the remaining tiles form the arrowed-square
tiling, whose squares are twice as large as the cells of the chair tiling, the
marker lying in the corner to which the arrow points. In the other direction,
remove the marker from every tile, obtaining a chair; the holes left are chairs as
well, since at every vertex either no arrow or exactly three arrows enter, and we
fill them.
%In one direction, mark each unit square of a chair tiling by the corner at the
%reflex (inner) vertex of the chair it belongs to; in the other, the three marked
%squares whose arrows meet at a common vertex, together with the empty fourth
%square of that $2\times2$ square, recover one chair. Both maps are local, so the two
%families of substitution tilings correspond to each other under this bijection.
\end{remark}

\begin{figure}[ht]
\begin{center}
\includegraphics{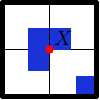}
\end{center}
\caption{The arrowed-square supertile of level $1$ obtained from a unit square
with arrow $(1,-1)$. Three of its arrows point to its center $X$, the fourth one
is the arrow $(1,-1)$ and points to a corner,
and no arrow points to the midpoint of a side.}
\label{pic-a1}
\end{figure}

\begin{figure}[ht]
\begin{center}
\includegraphics{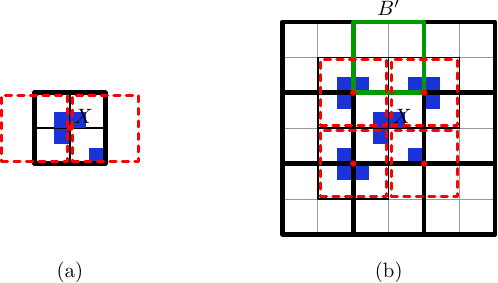}
\end{center}
\caption{An attempt to include the four unit squares of a level-$1$
supertile of a partition $P$ (drawn with solid lines) into level-$1$ supertiles
of another partition $P'$ (drawn with dashed lines); see the proof of
Lemma~\ref{lem-uc}. (a)~$P'$ is shifted by $(1,0)$ with respect to $P$: then $X$
is the midpoint of the common side of two supertiles of $P'$, although three
arrows point to it. (b)~$P'$ is shifted by $(1,1)$: the supertiles of $P'$ are centered at the
corners of the supertile $B$ of $P$ of Fig.~\ref{pic-a1}; the markers outside $B$
are forced by $P'$. The square $B'$ of $P$ north of $B$ then has two arrows
pointing to its corners, so it is not a level-$1$ supertile.}
\label{pic-a1shift}
\end{figure}

\begin{lemma}[Unique composition property of the chair substitution]\label{lem-uc}
Let $D$ and $D'$ be tilings of the plane by arrowed unit squares. If
$\tau D=\tau D'$, then $D=D'$.
\end{lemma}
\begin{proof}
Let $C=\tau D=\tau D'$. The arrowed-square supertiles
$\tau(u)$ of level $1$, $u\in D$, form a partition $P$ of $C$ into
$2\times2$ squares, and the supertiles $\tau(u')$, $u'\in D'$, form a
partition $P'$ of the same kind. Say that the arrow of a unit square
\emph{points to} a vertex $X$ if it points to the corner $X$ of this square. In a
level-$1$ supertile $\tau(u)$ three arrows point to its center, the fourth
arrow is the arrow of $u$ and points to a corner, and no arrow points to the
midpoint of a side (Fig.~\ref{pic-a1}).

If $P=P'$, then $D=D'$, since the arrow of $u$ is the only arrow of
$\tau(u)$ that does not point to its center. Assume that $P\ne P'$. Then the corners of the
supertiles of $P'$ are obtained from those of $P$ by a shift by $(1,0)$, $(0,1)$
or $(1,1)$ modulo $2\mathbb Z^2$.

In the first two cases (Fig.~\ref{pic-a1shift}(a)) the center $X$ of a
supertile of $P$ is the midpoint of the common side of two supertiles of $P'$,
so no arrow points to $X$. But three arrows point to $X$.

In the third case (Fig.~\ref{pic-a1shift}(b)) let $B$ be a supertile of
$P$. Its corners are the centers of four supertiles of $P'$, each containing one
unit square of $B$. Three unit squares of $B$ point to the center of $B$, that
is, away from the center of the supertile of $P'$ containing them; so each of
them is the quarter of its supertile of $P'$ whose arrow does not point to the
center, and the other three unit squares of this supertile of $P'$ point to its
center, which is a corner of $B$. Two of these three corners of $B$ are the
endpoints of a side of $B$ (in Fig.~\ref{pic-a1shift}(b), $B$ is the supertile of
Fig.~\ref{pic-a1}, and this is the northern side). The supertile $B'$ of $P$ adjacent to $B$ along
this side contains a unit square of each of the two corresponding supertiles of
$P'$, and these two unit squares point to the endpoints of the side, which are
corners of $B'$. But in a level-$1$ supertile only one arrow points to a corner.
\end{proof}

\paragraph{A map onto the chair tilings.}
In the rest of this section we denote tilings by square tiles by $S$ (with
primes and indices), tilings by triangles by $T$, and chair tilings (in the
arrowed-square form) by $C$.
Recall that every square tile in our tilings is a unit square whose center is a
$C_4$ crossing. The $C_4$ crossing has a distinguished oriented diagonal, its
\emph{axis} --- the pair of equally oriented, equally colored arrows passing
through the center (the purple arrow in Fig.~\ref{pic0}). As an oriented arrow
through the center, the axis could a priori point in any of the \emph{eight}
directions that are multiples of $45^\circ$ --- the directions taken by the
arrows of our tilings. In a given correct tiling, however, only \emph{four} of them
occur: the axis of a $C_4$ crossing runs along a diagonal of its
square and hence points to one of the four corners, and all square tiles of a
correct tiling are rotated by the same angle (Section~\ref{s2}). For a tiling of
the usual grid these are the directions $45^\circ+90^\circ k$. We use this restriction
implicitly. Define the map
$$
\varphi : (\text{square tile}) \longmapsto (\text{unit square with the diagonal
arrow equal to its } C_4 \text{ axis}),
$$
that is, $\varphi$ forgets all colors of arrows and triangles, forgets the
orientations of the sides of the tile, and keeps only the oriented
diagonal of each tile.

\begin{remark}\label{rem-eight}
The four available axis directions are not absolute; they depend on the tiling.
Recall from Section~\ref{s2} that the square tiles of $\sigma S$ are turned by
$45^\circ$ with respect to those of $S$. Hence their $C_4$ axes point in the
\emph{complementary} four directions; for a tiling $S$ of the usual grid these
are the multiples of $90^\circ$. A second decomposition restores the
original four. Thus along the substitution hierarchy all eight directions occur,
four at each parity of the level.
\end{remark}

\begin{figure}[ht]
\begin{center}
\includegraphics{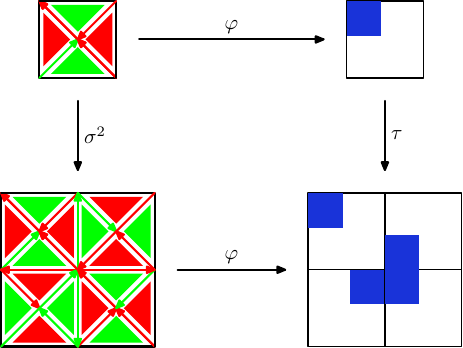}
\end{center}
\caption{The identity $\varphi(\sigma^2S)=\tau\varphi(S)$ of
Remark~\ref{rem-commute} for the first tile $\Box$ of Fig.~\ref{pic46}.
Left: the tile and its image under $\sigma^2$, which
consists of four square tiles forming a $2\times2$ square. The black sides may have any colors
and orientations; they are inherited from the sides of the tile. Right: the
images under $\varphi$, related by the arrowed-square substitution $\tau$.}
\label{pic-hatsigma}
\end{figure}

\begin{remark}\label{rem-commute}
Recall that $\tau$ denotes the arrowed-square substitution
(Fig.~\ref{pic-arrowsub}). For every tiling $S$ by square tiles
$$
\varphi(\sigma^2 S)=\tau\,\varphi(S)
$$
(Fig.~\ref{pic-hatsigma}); that is, two steps of our substitution correspond under $\varphi$ to
one step of the arrowed-square substitution. To see this, recall that
$\sigma^2\Box$ consists of four square tiles forming a $2\times2$ square
(Fig.~\ref{pic-hatsigma}); inspecting the substitution (Fig.~\ref{pic5}) shows
that the $C_4$ axes of these four are precisely the arrows that the arrowed-square
substitution assigns to the axis of $\Box$: the square in the corner to which the
axis of $\Box$ points has the same axis, and the other three have axes pointing
to the center of $\sigma^2\Box$. In particular they depend only on the axis of
$\Box$, not on the colors and edge orientations that $\varphi$ discards. This is
the identity $\varphi(\sigma^2\Box)=\tau(\varphi(\Box))$ for a single tile;
since both sides act tile by tile and commute with unions of fragments, it holds
for \emph{every} tiling $S$. The proof of Theorem~\ref{th-chair} below rests on
this identity.

Note that here $S$ is any tiling, not necessarily a substitution tiling.
If $S$ is a correct (= substitution) tiling, then
we can define also a single step of the decomposition $\sigma$
so that
$$
\pi(\sigma S)=\sigma\,\pi(S)
$$
(Fig.~\ref{pic-sigma-one}).
\begin{figure}[ht]
\begin{center}
\includegraphics{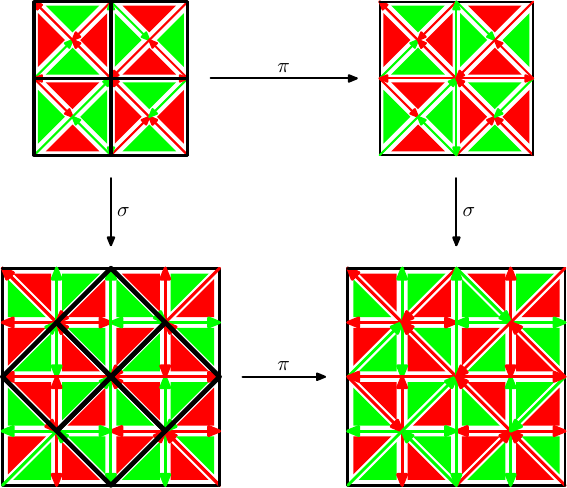}
\end{center}
\caption{The identity $\pi(\sigma S)=\sigma\pi(S)$ of Remark~\ref{rem-commute}
for the fragment $S=\sigma^2\Box$ of Fig.~\ref{pic-hatsigma}. Top: $S$ (left, its
four square tiles outlined) and $\pi(S)$ (right). Bottom: $\sigma S$ (left) and
$\sigma\pi(S)$ (right). The square tiles of $\sigma S$ are turned by $45^\circ$
and centered at the midpoints of the sides of the tiles of $S$; the four of them
lying entirely in $\sigma\pi(S)$ are outlined.}
\label{pic-sigma-one}
\end{figure}
Indeed, let $S$ be a correct tiling of the plane by square tiles.
Then the tiling $\sigma\pi(S)$ is correct (Lemma~\ref{lem-super}), hence by Lemma~\ref{lem-cut}(b)
%decomposition of a correct
%%substitution 
%tiling $T$ is a correct 
%%substitution 
%tiling. 
%%Indeed, a fragment of $\sigma T$ lies in $\sigma F$ for a fragment $F$
%%of $T$, and $F\subset\sigma^mt$ gives $\sigma F\subset\sigma^{m+1}t$.
%Hence for a correct tiling $S$ 
its triangles can be grouped into square tiles
% , so
%$\sigma\pi(S)$ is a correct triangle tiling, and
%$\sigma S$ is a correct tiling by square tiles 
turned by $45^\circ$. 
The resulting tiling by square tiles is by definition $\sigma S$.
Applying $\sigma$ twice, we obtain $\sigma^2S$, and
$\varphi(\sigma^2S)=\tau\varphi(S)$. In this sense the decomposition $\sigma$
is a square root of the arrowed-square substitution $\tau$, and hence of the chair
substitution.

Later we will show that $\sigma$ induces no map on the chair tilings (Remark~\ref{rem-nodescend}).
\end{remark}

\paragraph{The hierarchy of a correct tiling.}
Let $S$ be a correct tiling by square tiles; we choose
coordinates in which its tiles are the unit squares with vertices in
$\mathbb Z^2$. The \emph{edges} of $\pi(S)$ are the sides of its triangles, that
is, the sides of the tiles of $S$ and the halves of their diagonals. By
Theorem~\ref{th2} all compositions $\sigma^{-m}\pi(S)$ exist and are correct; the
sets $\sigma^m t$, $t\in\sigma^{-m}\pi(S)$, are called the \emph{triangle
supertiles of level $m$} of $S$. For every $n$ the composition $\sigma^{-2n}S$ is a correct
tiling by square tiles with $\pi(\sigma^{-2n}S)=\sigma^{-2n}\pi(S)$
(Section~\ref{s2} and Lemma~\ref{lem-cut}). The tilings $\sigma^{2n}Q$,
$Q\in\sigma^{-2n}S$, are squares formed by $2^n\times2^n$ tiles of 
$S$%(theypartition $S$, since $S=\sigma^{2n}(\sigma^{-2n}S)$)
; we call
them the \emph{square supertiles of level $2n$} of $S$ and denote them by $S_{2n}$, and we call the
diagonal of $\sigma^{2n}Q$ carrying the $C_4$ axis of $Q$ its \emph{axis
diagonal}. Thus a square supertile $S_{2n}$ has the size of the triangle
supertiles of level $2n$. Each square supertile of level $2n+2$ consists of four
square supertiles of level $2n$, its \emph{quarters}; hence the corners of the
square supertiles of level $2n$ form a grid $(a_n,b_n)+2^n\mathbb Z^2$, where
$a_{n+1}\equiv a_n$ and $b_{n+1}\equiv b_n \pmod{2^n}$. Square supertiles
are defined only for even levels. %: a single decomposition does not act on single square tiles. 

\begin{figure}[ht]
\begin{center}
\includegraphics{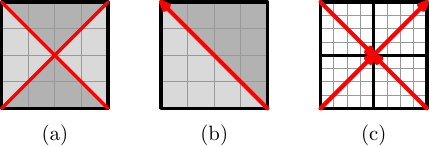}
\end{center}
\caption{(a) A square supertile $S_{2n}=\sigma^{2n}Q$ and its four triangle
supertiles $\sigma^{2n}t$ of level $2n$, where $t$ runs over the four triangles
of $\pi(Q)$. (b) Its two triangle supertiles of level $2n+1$, cut off by the axis
diagonal (the arrow, which is the $C_4$ axis of $Q$).
%: the composition of
%$\sigma^{-2n}\pi(S)$ merges the two triangles of $\pi(Q)$ lying on the same side of the
%axis of $Q$, so these supertiles are the halves cut off by the axis diagonal (the arrow). 
(c) A square supertile $S_{2n+2}$: the axis diagonals of its four quarters 
%(arrows, as given by the arrowed-square substitution) 
form its two diagonals.}
\label{pic-levels}
\end{figure}

\begin{definition}\label{def-infsup}
An edge of $\pi(S)$ is \emph{singular} if for every $m$ it lies on the boundary
of a triangle supertile of level $m$ of $S$. We denote by $B(S)$ the union of the
singular edges.
\end{definition}

\begin{lemma}\label{lem-bd}
The union of the boundaries of the triangle supertiles of level $m$ of $S$ is
formed
\begin{equation}\label{eq-bd}
\begin{aligned}
&\text{for } m=2n: &&\text{by the sides and the diagonals of the square supertiles $S_{2n}$,}\\
&\text{for } m=2n+1: &&\text{by the sides and the axis diagonals of the square supertiles $S_{2n}$.}
\end{aligned}
\end{equation}
%In particular, f
For any diagonal edge $e$ of $\pi(S)$ the following are equivalent:\\
(a) $e\in B(S)$, that is, $e$ lies on the boundary of triangle supertiles of all levels of $S$,\\ 
(b) $e$ lies on the axis diagonal of a square supertile of every level  of $S$,\\
(c) $e$ lies on a diagonal of a square supertile of every level  of $S$.
\end{lemma}
\begin{proof}
The case
$m=2n$ is seen from Fig.~\ref{pic-levels}(a): the triangle supertiles of level $2n$ are the sets
$\sigma^{2n}t$ for the four triangles $t$ of $\pi(Q)$, $Q\in\sigma^{-2n}S$, that
is, the quarters of $S_{2n}=\sigma^{2n}Q$ cut off by its two diagonals.

The case $m=2n+1$ is seen from Fig.~\ref{pic-levels}(b).
In the composition of
$\sigma^{-2n}\pi(S)=\pi(\sigma^{-2n}S)$ the two triangles of $\pi(Q)$ lying on the
same side of the axis of $Q$ merge into one triangle, whose hypotenuse is formed
by the two arrows of the axis (see the proof of Theorem~\ref{th2}). Hence a
triangle supertile of level $2n+1$ is a half of $S_{2n}$ cut off by its axis
diagonal: the axis diagonal remains on the boundary, while the other diagonal
becomes interior.

%By
%\eqref{eq-bd}, a diagonal edge lies on the boundary of the triangle supertiles of
%level $2n$ if and only if it lies on a diagonal of some $S_{2n}$, and on the
%boundary of those of level $2n+1$ if and only if it lies on the axis diagonal of
%some $S_{2n}$. 
The last claim: by \eqref{eq-bd} for odd $m$, we have (a) $\Rightarrow$ (b). Obviously, (b) $\Rightarrow$ (c).
To show (c) $\Rightarrow$ (a),  it suffices to show: if a diagonal edge lies
on a diagonal of some $S_{2n+2}$, then it lies on the axis diagonal of some
$S_{2n}$. Every triangle supertile of level $2n+2$ is the union of two
triangle supertiles of level $2n+1$, so the boundaries of the triangle supertiles
of level $2n+2$ are contained in those of level $2n+1$. By \eqref{eq-bd}, applied
to $m=2n+2$ and to $m=2n+1$, this inclusion says exactly what we need: every
diagonal of a square supertile $S_{2n+2}$ is covered by axis diagonals of square
supertiles $S_{2n}$ (a diagonal edge cannot lie on a side). Fig.~\ref{pic-levels}(c)
shows how this happens: the diagonal of $S_{2n+2}$ is the union of the axis
diagonals of the two quarters of $S_{2n+2}$ lying on it; note that, unless it is the
axis diagonal of $S_{2n+2}$ itself, it is not contained in the axis diagonal of
any square supertile of level $2n+2$.
%; in fact it is the union of the axis diagonals of the two quarters of
%Every triangle supertile of level $2n+2$ is the union of two triangle
%supertiles of level $2n+1$, so the boundaries of the former are contained in the
%boundaries of the latter. By \eqref{eq-bd} this means that every diagonal of a
%square supertile $S_{2n+2}$ is covered by axis diagonals of square supertiles
%$S_{2n}$; in fact it is the union of the axis diagonals of the two quarters of
%$S_{2n+2}$ lying on it. 
%(Note that a diagonal of $S_{2n+2}$ is in general not
%contained in the axis diagonal of a single square supertile of the same level.)
\end{proof}

\begin{lemma}\label{lem-sing}
For every correct tiling $S$ by square tiles the set $B(S)$ is of one of the
following three kinds:
\begin{itemize}
\item[(a)] empty;
\item[(b)] a single line, horizontal, vertical or diagonal (of slope $\pm1$);
\item[(c)] the union of four lines --- a horizontal, a vertical and two
diagonal ones --- passing through a common vertex $X_0$ of the grid.
\end{itemize}
\end{lemma}
\begin{proof}
The sides of the square supertiles of level $2n$ lie on the lines $x\in
a_n+2^n\mathbb Z$ and $y\in b_n+2^n\mathbb Z$, and their diagonals on the lines
$x-y\in (a_n-b_n)+2^n\mathbb Z$ and $x+y\in (a_n+b_n)+2^n\mathbb Z$; each of
these four families of lines decreases as $n$ grows. A horizontal or vertical
edge of $\pi(S)$ can lie only on a side of a square supertile and a diagonal
edge only on a diagonal, so by Lemma~\ref{lem-bd} a vertical
edge is singular iff its line $x=a$ satisfies $a\equiv a_n \pmod {2^n}$ for all
$n$, and similarly for the other three directions. Two numbers congruent modulo
every $2^n$ are equal, so $B(S)$ contains at most one line of each of the four
directions. If $B(S)$ contains the lines $x=a$ and $y=b$, then $a-b\equiv
a_n-b_n$ and $a+b\equiv a_n+b_n$ modulo every $2^n$, so it contains both
diagonals through $(a,b)$. If it contains $x=a$ and the diagonal $x-y=d$, then
$b=a-d$ satisfies $b\equiv b_n \pmod{2^n}$ for all $n$, so it contains the line
$y=b$ and we are in the previous situation; the same applies to any two lines of
different directions one of which is horizontal or vertical. Finally, if it
contains the diagonals $x-y=d$ and $x+y=s$, then $d+s\equiv 2a_n\pmod{2^n}$ for
all $n$; hence $a=(d+s)/2$ is an integer and $a\equiv a_n\pmod{2^{n-1}}$ for all
$n$, so $B(S)$ contains the line $x=a$. Thus $B(S)$ contains either at most one
line or all four lines through a point, and this point is a vertex of the grid.
\end{proof}

\begin{theorem}\label{th-chair}
The map $\varphi$ sends every correct tiling by square tiles to a chair tiling,
and it is \emph{onto}: every chair tiling is the image of some correct tiling.
If $\varphi(S)=\varphi(S')$, then for every level the square supertiles of $S$
and of $S'$ occupy the same squares, and the triangle supertiles of $S$ and of
$S'$ occupy the same triangles (their decorations may differ); hence
$B(S)=B(S')$, and $\pi(S)$ and $\pi(S')$ differ only in the colors and
orientations of edges lying in $B(S)$. More precisely, the number of
$\varphi$-preimages of $C=\varphi(S)$ is
\begin{itemize}
\item[(i)] one, if $B(S)=\varnothing$;
\item[(ii)] two, if $B(S)$ is a single diagonal line; the preimages differ in
the color of this line;
\item[(iii)] four, if $B(S)$ is a single horizontal or vertical line; the
preimages differ in the color and in the orientation of this line;
\item[(iv)] four, if $B(S)$ consists of four lines through a vertex $X_0$ and
exactly one of the four unit squares of $C$ at $X_0$ carries an arrow pointing
away from $X_0$ (Fig.~\ref{pic-iv}, top); the preimages differ in the color of
the axis of the crossing at $X_0$ and in the distribution of the two colors
between the two diagonal rays entering $X_0$;
\item[(v)] sixteen, if $B(S)$ consists of four lines through a vertex $X_0$ and
all four unit squares of $C$ at $X_0$ carry arrows pointing away from $X_0$
(Fig.~\ref{pic-iv}, bottom); the axis of the crossing at $X_0$ is then horizontal
or vertical, and the preimages differ in which of the two it is, in the
orientation and the color of the axis, and in the distribution of the two colors
between the two rays entering $X_0$.
\end{itemize}
If $B(S)$ consists of four lines through a vertex, then one of the cases (iv),
(v) occurs. In particular, every chair tiling has at most sixteen
preimages.

 All these cases occur.
\end{theorem}

\begin{proof}
%By Corollary~\ref{cor-perfect} the correct tilings are exactly the substitution tilings. 
The map $\varphi$ is local: the arrowed square it places at a position
depends only on the square tile of $S$ there. Hence $\varphi$ commutes with
translations and sends every fragment of $S$ to the fragment of $\varphi(S)$ over the
same region. Applying the identity of Remark~\ref{rem-commute} $n$ times gives
$$
\varphi\big(\sigma^{2n}\Box\big)=\tau^n\big(\varphi(\Box)\big)
$$
for every square tile $\Box$. Thus $\varphi$ carries square supertiles of level
$2n$ onto arrowed-square supertiles of level $n$, and --- as $\varphi(\Box)$ runs
over the four arrows --- onto all of them. In particular, for $Q\in\sigma^{-2n}S$ the
arrow of the arrowed-square supertile $\varphi(\sigma^{2n}Q)$ is the $C_4$ axis of
$Q$.

\smallskip
\noindent\emph{The image of a substitution tiling is a substitution tiling.}
Let $P$ be a finite fragment of $\varphi(S)$. Then $P=\varphi(P')$ for the
fragment $P'$ of $S$ over the same region. By Corollary~\ref{cor-perfect-sq},
$P'$ lies in a square supertile $\sigma^{2n}\Box$, so
$P=\varphi(P')$ lies in $\varphi(\sigma^{2n}\Box)=\tau^n\varphi(\Box)$,
an arrowed-square supertile. Thus $\varphi(S)$ is a chair tiling.

\smallskip
\noindent\emph{Every chair tiling has a preimage.}
Let $C$ be a chair tiling, and exhaust the plane by fragments $P_1\subset
P_2\subset\cdots$ of $C$. Call a tiling $L$ by square tiles a \emph{lift} of
$P_k$ if it covers the same region as $P_k$, satisfies $\varphi(L)=P_k$, and is a
fragment of a tiling $\sigma^{2n}\Box$ for some $n\ge0$ and some square tile
$\Box$. Every $P_k$ has a lift: $P_k$ lies in an arrowed-square supertile
$\tau^n(c)$; if $\Box$ is a square tile with axis $c$, then
$\tau^n(c)=\tau^n\varphi(\Box)=\varphi(\sigma^{2n}\Box)$, so the corresponding
part of $\sigma^{2n}\Box$ is a lift of $P_k$. The restriction of a lift of $P_k$
to the region of $P_{k-1}$ is a lift of $P_{k-1}$. Form the graph whose nodes are
a root and the lifts of all $P_k$; the root is joined to all lifts of $P_1$, and
every lift of $P_k$, $k\ge2$, is joined to its restriction to the region of
$P_{k-1}$. Every node other than the root has exactly one neighbor on the
previous level, so this graph is a tree; the restrictions of a lift of $P_k$ to
the regions of $P_i$, $i<k$, are the nodes on the path from it to the root. Over
a fixed region there are only finitely many tilings by square tiles, so the tree
is finitely branching; being infinite, it has an infinite branch (König's
lemma~\cite{koenig}). The lifts along this branch extend each other, and their
union is a tiling $S$ with $\varphi(S)=C$. Every finite fragment of $S$ lies in
some lift, hence in a tiling $\sigma^{2n}\Box$; so $S$ is a substitution tiling
by construction, and hence correct by Corollary~\ref{cor-perfect-sq}.

\smallskip
\noindent\emph{Preimages of $C$ agree off $B(S)$.}
Let $\varphi(S)=C$ and $n\ge0$. Since $S=\sigma^{2n}(\sigma^{-2n}S)$, applying the identity of Remark~\ref{rem-commute} $n$
times we obtain
$$
C=\varphi(\sigma^{2n}(\sigma^{-2n}S))=\tau^n\varphi(\sigma^{-2n}S).
$$
By Lemma~\ref{lem-uc}, applied $n$ times, there is only one tiling $D$ by arrowed
squares with $\tau^nD=C$. Hence $\varphi(\sigma^{-2n}S)=\tau^{-n}C$ is
uniquely determined by $C$ and $n$, and does not depend on the choice of the
preimage $S$. The images under $\varphi$ of the square supertiles
$\sigma^{2n}Q$, $Q\in\sigma^{-2n}S$, of level $2n$ of $S$ are the arrowed-square
supertiles $\tau^n(\varphi(Q))$, so the partition of the plane into
square supertiles of level $2n$ is also determined by $C$. So for all preimages of $C$ the
square supertiles of level $2n$ occupy the same squares, and their axis
diagonals, given by the arrows of $\varphi(\sigma^{-2n}S)$, are the same; hence the triangle
supertiles of each level occupy the same triangles, and the set $B(S)$ is the
same. Moreover, the $C_4$ axes of the square tiles of
$\sigma^{-2n}S$ are the arrows of $\varphi(\sigma^{-2n}S)$, and by rule~(4) the axis of a square tile
determines the colors of its four triangles. Hence the colors of all triangles of
all compositions $\sigma^{-2n}\pi(S)=\pi(\sigma^{-2n}S)$ are determined by $C$, and by
Lemma~\ref{lem-inside} so is every edge of $\pi(S)$ lying inside a triangle
supertile of level $2n$. An edge of $\pi(S)$ that is not singular lies inside a
triangle supertile of some level $m$, and then inside the triangle supertile of
level $2n\ge m$ containing it. Thus we have proved:
\begin{itemize}
\item[$(\ast)$] if $\varphi(S')=C$, then $\pi(S')$ coincides with $\pi(S)$
except possibly for the colors and orientations of edges lying in $B(S)$.
\end{itemize}
This proves~(i).

\smallskip
We now analyze the edges of $\pi(S)$ lying in $B(S)$. It suffices to consider
the local rules for tilings by square tiles: the crossings at the centers of the
square tiles are built into the tile set, so we only have to look at the
crossings at the vertices of $S$ (which are $C_8$ crossings) and at the tile set.
Two observations will be used repeatedly.
\begin{itemize}
\item[(O1)] If $X$ is a vertex of $S$ and the two edges at $X$ lying on some line
through $X$ are oriented the same way (one enters $X$, the other leaves it), then
this pair is the axis of the crossing at $X$, because in a legal crossing the
perpendicular pair enters the center and the four remaining arrows leave it; in
particular the two edges then have the same color.
\item[(O2)] The legal crossings are closed under changing the color of the axis,
under reversing the axis, and under interchanging the colors of the two entering
arrows (the colors of the four leaving arrows depend only on the line of the
axis); this is seen from Fig.~\ref{pic44} and its rotations. Likewise, changing
the color of the axis of a square tile, or reversing the axis, gives again a
square tile of our set (compare the first tile of Fig.~\ref{pic46} with the third
and with the second one).
\end{itemize}
Observation (O2) is needed to show that the candidates for preimages found below
are indeed preimages: a preimage must be a correct tiling, so we have to check
that changing the colors or orientations of the edges on $B(S)$ preserves the
local rules.

\smallskip
\noindent\emph{(ii) $B(S)$ is a single diagonal line $\ell$}, say of slope $1$.
Let $Q$ be a square tile of $S$ centered on $\ell$ and let $Q=S_0\subset
S_2\subset S_4\subset\cdots$ be the square supertiles of levels $0,2,4,\dots$
containing it. Every edge on $\ell$ is singular, so by Lemma~\ref{lem-bd}(b)
the axis diagonal of every $S_{2n}$ lies on $\ell$; hence the axis diagonal of
$S_{2n}$ is a half of the axis diagonal of $S_{2n+2}$. The edges of $\pi(S)$ on
the axis diagonal of $S_{2n}=\sigma^{2n}Q_n$ are obtained by the substitution from
the two legs of $\pi(Q_n)$ forming the axis of $Q_n$, and the substitution keeps
the orientations of sides; so all these edges are oriented as the axis of $Q_n$.
Therefore the edges on the axis diagonals of all $S_{2n}$ are oriented in one
and the same direction $c=\pm(1,1)$. The axis diagonals of
the $S_{2n}$ exhaust $\ell$: otherwise they would have a common endpoint $X$ on
$\ell$ for all large $n$, and $X$ would be a corner of square supertiles of all
levels, so the horizontal line through $X$ would be singular too. Hence all edges of
$\pi(S)$ on $\ell$ are oriented in the direction $c$. In particular, all square
tiles centered on $\ell$ have the axis in the direction $c$ (in a square tile
the two edges on the axis are oriented alike, whereas the two edges on the other
diagonal both enter the center), and by (O1) the
crossing at every vertex of $S$ on $\ell$ has its axis on $\ell$; so all edges of
$\pi(S)$ on $\ell$ have one and the same color (two consecutive edges on $\ell$
meet either at a vertex of $S$ or at the center of a square tile, where they form
the axis). Now let $S'$ be any preimage of $C$. Its square supertiles occupy
the same squares as those of $S$ and have the same axis diagonals with the same
orientations (they are given by the arrows of $\varphi(\sigma^{-2n}S')=\tau^{-n}C$),
so the same
argument applies to $S'$: all edges of $\pi(S')$ on $\ell$ are oriented in
the direction $c$ and have one and the same color. By $(\ast)$, applied to $S'$,
$\pi(S')$ coincides with $\pi(S)$ off $\ell$. Hence $S'$ is either $S$ or
$S$ with the color of all edges on $\ell$ changed. Conversely, by (O2) the latter
tiling is correct, and its image is $C$, as $\varphi$ forgets colors; so $C$ has
exactly two preimages.

\smallskip
\noindent\emph{(iii) $B(S)$ is a single vertical line $\ell$} (the horizontal
case is symmetric). Let $X$ be a vertex on $\ell$. As the horizontal line through
$X$ is not singular, there is a largest $j$ such that $X$ is a corner of the
square supertiles of level $2j$. Then $X$ is the midpoint of the common vertical
side of two square supertiles $\sigma^{2j+2}Q_W$ and $\sigma^{2j+2}Q_E$ of level
$2j+2$. 
%Let $t$ be one of the four triangles of $\pi(Q)$ for a square tile $Q$. The
%substitution applied to $t$ creates a new edge, the height of $t$, which runs from
%the vertex of the right angle of $t$, the center of $Q$, to the midpoint of the
%hypotenuse of $t$, a side of $Q$, and is oriented away from the center of $Q$.
%Further decompositions do not change the orientation of existing edges (they
%only cut edges into equally oriented halves). Hence in the square supertile
%$\sigma^{2j+2}Q$, whose triangles form $\bigcup_t\sigma^{2j+2}t$, all edges on the four segments joining
%its center to the midpoints of its sides are oriented away from the center. 
The
horizontal arrows at $X$ lie on the segments joining the centers of
$\sigma^{2j+2}Q_W$ and $\sigma^{2j+2}Q_E$ to the midpoints of their sides; these
segments are the heights of the triangles of $\pi(Q_W)$ and $\pi(Q_E)$ cut into
pieces, and heights are oriented from the vertex of the right angle, that is, away
from the center (see the lower left picture in Fig.~\ref{pic-hatsigma}). So the
two horizontal arrows at $X$ both enter $X$. In a legal crossing the two arrows
entering the center are perpendicular to the axis; hence the axis of the crossing
at $X$ is vertical, that is, the two edges on $\ell$ at $X$ have the same color
and the same orientation. As this holds at every vertex on $\ell$, all edges of
$\pi(S)$ on $\ell$ have one and the same color and one and the same
orientation. (Here, unlike in (ii), we do not use (O1): the orientations of the
edges on $\ell$ are not given by $C$, and the axis is found from the entering
pair instead.) Now let $S'$ be any preimage of $C$. Its square supertiles occupy the same
squares, so the same argument applies to $S'$, and all edges of $\pi(S')$ on
$\ell$ have one and the same color and one and the same orientation. By $(\ast)$,
applied to $S'$, $\pi(S')$ coincides with $\pi(S)$ off $\ell$. Hence $S'$
is one of the four tilings obtained from $S$ by changing the color and/or the
orientation of all edges on $\ell$. Conversely, by (O2) all four are correct,
and their image is $C$, as $\varphi$ forgets the colors and the orientations of
horizontal and vertical edges; so $C$ has exactly four preimages.

\smallskip
\begin{figure}[!htbp]
\begin{center}
\includegraphics[scale=0.85]{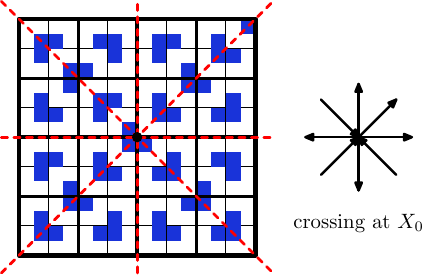}

\medskip
\includegraphics[scale=0.85]{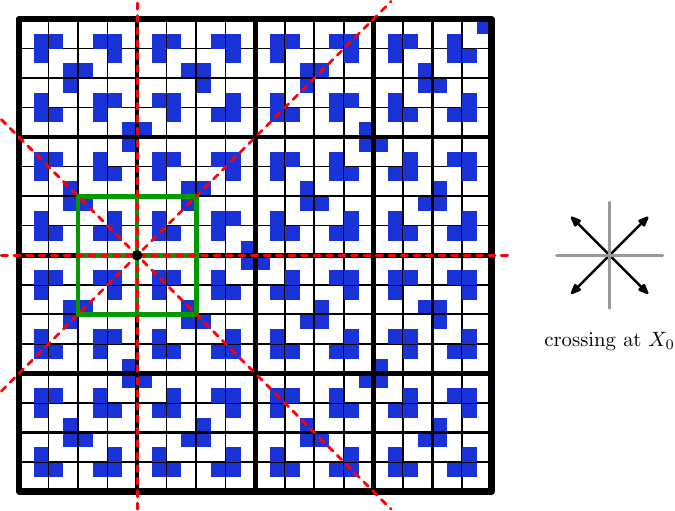}
\end{center}
\caption{Cases (iv) (top) and (v) (bottom) of Theorem~\ref{th-chair}. Dashed: the
singular lines through $X_0$ of the limit tiling $C$ constructed in the paragraph
``All cases occur'' of the proof; on the right: the arrows of the crossing at $X_0$
forced by $C$ (grey: not determined). Top: $\tau^3(1,1)$ with $X_0$ at its center;
by property (P) below it is the central part of every $\tau^n(1,1)$, $n\ge3$, hence a fragment of
$C$. Bottom: $\tau^4(1,1)$ with $X_0=(4,8)$; the four level-$1$ supertiles at $X_0$
(green) point away from $X_0$, and so do the corresponding ones at $p_n$ in
$\tau^n(1,1)$ for all $n\ge4$. Only the green square is a fragment of $C$; outside it
the diagonal edges on the dashed lines are, in general, not on axis diagonals.}
\label{pic-iv}
\end{figure}

\noindent\emph{(iv), (v): $B(S)$ consists of four lines through $X_0$.} Let 
%four quadrants $q$ at $X_0$ let 
$L_q$ denote the ray with origin $X_0$ in the direction $q$ and 
$S^q_{2n}$  the  square supertile of $S$ of level $2n$ 
with corner $X_0$ in the quadrant $q$ at $X_0$. 
We claim that all edges of $\pi(S)$ lying on $L_q$ have the same color and orientation.
By Lemma~\ref{lem-bd}, for $q=(\pm1,\pm1)$
the ray $L_q$ is the union of the axis diagonals of the supertiles $S^q_{2n}$
(the diagonal of $S^q_{2n}$ issuing from $X_0$ lies on $L_q\subset B(S)$, so it is
the axis diagonal); in particular the arrow of the unit square $S^q_0$ is $\pm q$. The axis diagonal of $S^q_{2n}$
is obtained by cutting into halves that of $S^q_{2n+2}$ keeping 
its color and orientation. Thus axis diagonals of all $S^q_{2n}$ 
have the same color and orientation. 
Every edge of  $\pi(S)$ lying on $L_q$ lies on an axis diagonal of some  $S^q_{2n}$
and can be obtained from it by several cuttings in halves. Hence the claim.

For horizontal and vertical rays $L_q$ the argument is similar; this time we use
the sides of the supertiles $S^{q'}_{2n}$ lying on $L_q$, where $q'$ is one of the
two diagonal directions adjacent to $q$, in place of the axis diagonals: the side
of $S^{q'}_{2n}$ on $L_q$ is a half of the side of $S^{q'}_{2n+2}$ on $L_q$, with the
same color and orientation.

The claim holds for every preimage $S'$ of $C$ as well, since the argument uses
only the positions of the square supertiles and their axis diagonals, which are
the same for all preimages. Moreover, at every vertex $X\ne X_0$ of $S$ on a ray
$L_q$ the two edges of $L_q$ at $X$ are oriented the same way, so by (O1) they
form the axis of the crossing at $X$; hence, by (O2), changing the color of a
ray, or the orientation of a horizontal or vertical ray, preserves the local
rules at all vertices other than $X_0$.

By $(\ast)$ every preimage coincides with $S$ off
the eight rays, and by (O2) the preimages of $C$ are
exactly the tilings obtained from $S$ by choosing the colors and orientations of
the eight rays subject to these constraints and to the legality of the crossing at
$X_0$. In a legal crossing, either one diagonal passes through the center and the
other enters it, so that exactly one diagonal arrow leaves the center, or the
axis is horizontal or vertical and all four diagonal arrows leave it. Hence
%either exactly one $c^q$ equals $q$, or all of them do; that is, 
one of the cases (iv), (v) occurs. In case (iv) the axis is fixed, and the
preimages are determined by the color of the axis and by the distribution of the
two colors between the two entering diagonal rays: four preimages (the colors
and orientations of the horizontal and vertical rays, which leave $X_0$, are
then forced). In case (v) a preimage is determined by
whether the axis is horizontal or vertical, by its orientation, its color and the
distribution of the colors between the two entering rays: sixteen preimages.

\smallskip
\noindent\emph{All cases occur.} 
In the construction we have to
determine the arrows of $C$ at a given point of an arrowed-square supertile. For
this we use the following property of the arrowed-square substitution. Recall that
$\tau^n(c)$ is the arrowed-square supertile of level $n$ obtained from a unit
square with arrow $c$. It is the union of four arrowed-square
supertiles of level $n-1$, its \emph{quarters}: the quarter in the corner $c$ is
$\tau^{n-1}(c)$, and the other three quarters are $\tau^{n-1}$ of the arrows pointing to
the center of $\tau^n(c)$ (Fig.~\ref{pic-arrowsub}). We call $c$ the \emph{arrow} of
$\tau^n(c)$.
\begin{itemize}
\item[(P)] The arrow of the quarter of $\tau^n(c)$ lying in the corner $q$ is $-q$
(it points to the center of $\tau^n(c)$), except for $q=c$, when it is $q$.
\end{itemize}

Let $\Box$ be a square tile.
For $n=4,8,12,\dots$ consider the square supertile
$\sigma^{2n}\Box$, occupying the square $[0,2^n]^2$, and translate it so that a
chosen point $p_n\in\mathbb Z^2$ of this square is moved to the origin. The
points $p_n$ chosen below lie at distance at least $2^n/16$ from the boundary of
the square, so for every $r$ the translated supertiles cover the disc of radius
$r$ around the origin once $n$ is large. As in the
proof that every chair tiling has a preimage, König's lemma yields a tiling $S$
such that for every $r$ the fragment of $S$ within distance $r$ from the origin
coincides with the corresponding fragment of infinitely many of these translated
supertiles; $S$ is a substitution tiling, hence correct
(Corollary~\ref{cor-perfect-sq}).

Fix $k$. For $n\ge k$ the square supertiles of level $2k$ of
$\sigma^{2n}\Box$ are the squares of side $2^k$ with corners in $2^k\mathbb Z^2$,
so after the translation their corners form the grid $-p_n+2^k\mathbb Z^2$. For
the points $p_n$ chosen below, the residues of the coordinates of $p_n$ modulo
$2^k$ are the same for all $n\ge k+2$ (their binary expansions end with the same
$k$ digits). Hence the grid $-p_n+2^k\mathbb Z^2$ is the same for all $n\ge k+2$,
and, as compositions are determined locally, the corners of the square
supertiles of level $2k$ of $S$ form this grid; that is, in the notation of the
proof of Lemma~\ref{lem-sing} we may take $(a_k,b_k)=-p_n$ for any $n\ge k+2$. By
that proof, the vertical line $x=a$ is singular iff $a\equiv -x(p_n)
\pmod{2^k}$ for all $k$ and all $n\ge k+2$, where $x(p_n)$ is the first coordinate
of $p_n$. For fixed $a$ and large $k$ the binary expansion of $-a$ modulo $2^k$
consists, apart from its last few digits, only of zeros ($a\le0$) or only of ones
($a>0$). Thus some vertical line is singular iff the binary expansions of the
numbers $x(p_n)$ consist of zeros only, or of ones only, apart from their last
few digits and their first digit; and similarly for horizontal lines with the
second coordinate $y(p_n)$ and for the two diagonal directions with $x(p_n)-y(p_n)$
and $x(p_n)+y(p_n)$.

Let $s_n=(2^n-1)/3$ and $r_n=(2^n-1)/15$; their binary expansions are
$$
s_n=0101\dots01,\qquad r_n=0001\,0001\dots0001.
$$
For $p_n=(s_n,s_n)$ only the difference
$x(p_n)-y(p_n)=0$ has the required form, which gives type (ii); for $p_n=(2^{n-1},s_n)$ only
$x(p_n)=2^{n-1}=100\dots0$ does, which gives type (iii); for $p_n=(s_n,r_n)$ none of
$s_n$, $r_n$, $s_n-r_n=4r_n$ and $s_n+r_n=6r_n$ does, which gives type (i).
It remains to realize cases (iv) and (v) (Fig.~\ref{pic-iv}). Choosing $p_n=(2^{n-1},2^{n-1})$ or
$p_n=(2^{n-2},2^{n-1})$, we have $x(p_n)\equiv y(p_n)\equiv0\pmod{2^k}$ for
$n\ge k+2$, so the origin is a corner of the square supertiles of $S$ of every
level, and $B(S)$ consists of four lines through the origin. Which case
occurs depends on the four unit squares of $C=\varphi(S)$ at the origin. For
large $n$ they are the images under $\varphi$ of the four square tiles of
$\sigma^{2n}\Box$ at $p_n$, that is, the four unit squares at $p_n$ of
$\varphi(\sigma^{2n}\Box)=\tau^n(\varphi(\Box))=\tau^n(c)$, where $c$ is
the axis of $\Box$.

For $p_n=(2^{n-1},2^{n-1})$, the center of $\tau^n(1,1)$, apply (P) repeatedly: the
quarter of $\tau^n(1,1)$ in the corner $q$ has the arrow $q$ if $q=(1,1)$ and $-q$
otherwise, and so do its quarters at the center of $\tau^n(1,1)$ of all lower levels.
Hence exactly one of the four unit squares at the center, the one in the
quadrant $(1,1)$, has an arrow pointing away from the origin
(Fig.~\ref{pic-iv}, top). This is case~(iv).

For case (v) let the axis of $\Box$ be $c=(1,1)$ and
$p_n=(2^{n-2},2^{n-1})$. The four arrowed-square supertiles of level $1$ of
$\tau^4(1,1)$ with a corner at the point $(4,8)$ all have arrows pointing away
from this point (the green squares in Fig.~\ref{pic-iv}, bottom).
Applying the arrowed-square substitution $n-4$ times, we see that the four
arrowed-square supertiles of level $n-3$ of $\tau^n(1,1)$ with a corner at
$2^{n-4}(4,8)=p_n$ have arrows pointing away from $p_n$, and by (P) so do their
quarters at $p_n$ of all lower levels, down to the unit squares. So all four unit
squares of $C$ at the origin carry arrows pointing away from the origin; this is
case~(v).
\end{proof}

\begin{figure}[ht]
\begin{center}
\includegraphics{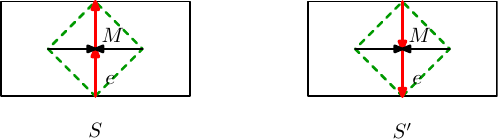}
\end{center}
\caption{Remark~\ref{rem-nodescend}. Two square tiles of $S$ (left) and of $S'$
(right) sharing a vertical edge $e$ on the line $\ell$; the tilings differ only
in the orientation of $e$. Dashed: the square tile of $\sigma S$, respectively of
$\sigma S'$, centered at the midpoint $M$ of $e$; its vertices are the ends of $e$
and the centers of the two tiles. The two halves of $e$ form its axis, and the two
heights (black) enter $M$. The axes of the two dashed tiles are opposite, so
$\varphi(\sigma S)$ and $\varphi(\sigma S')$ differ at $M$.}
\label{pic-nodescend}
\end{figure}

\begin{remark}\label{rem-nodescend}
The decomposition $\sigma$ induces no map on the chair tilings: there is no map
$\bar\sigma$ on chair tilings with
$$
\varphi(\sigma S)=\bar\sigma(\varphi(S))
$$
for all correct tilings $S$ by square tiles. Indeed, by Theorem~\ref{th-chair}(iii)
there are correct tilings $S\ne S'$ with $\varphi(S)=\varphi(S')$ that differ only
in the orientation of the edges on a vertical line $\ell$. In $\sigma S$
%(Section~\ref{s2}) 
the two halves of every edge $e$ on $\ell$ form the axis of the
square tile centered at the midpoint of $e$ (Fig.~\ref{pic-nodescend}). Hence these axes in $\sigma S$ and in
$\sigma S'$ are opposite, and $\varphi(\sigma S)\ne\varphi(\sigma S')$. Thus the
square root $\sigma$ of the chair substitution exists only on our correct tilings,
not on the chair tilings themselves.
\end{remark}

\paragraph{Relation to the Trilobite and Crab.}
Since the chair itself is not enforced by local rules, it is natural to look for
a family of tilings defined by local matching rules together with a local map of
this family onto the chair tilings;
Theorem~\ref{th-chair} shows that our family is such a family. It is
not the first one. Goodman-Strauss~\cite{gscrab} constructed the
\emph{Trilobite and Crab}, a pair of tiles whose matching rules force the chair
(L-tromino) substitution hierarchy; this fits the general framework of
\cite{gs} relating substitution tilings and matching rules. The two
constructions differ, however, in what exactly they enforce.

The Trilobite and Crab %, like the general construction of~\cite{gs}, 
force the
substitution \emph{hierarchy} rather than the substitution family itself. This
is the notion of enforcement used throughout~\cite{gslots}, where a set of tiles
is said to enforce a substitution system if every admitted tiling is carried by
a local map to a \emph{hierarchical} tiling, and where the Trilobite and Crab
are described as a pair of tiles each of whose tilings ``must be a marked up
hierarchical L-tiling''. Forcing the hierarchy is weaker than capturing the
substitution family, since rules of this kind admit additional non-substitution
tilings along fault lines~\cite{frank,gsrev}. For the Trilobite and Crab this is
already explicit in~\cite{gscrab}: Lemma~3 there asserts that a tiling by these
tiles can be recomposed into a substitution tiling only ``after a (possibly
infinite) series of shifts along concentric parallel $\gamma$ chains''. A
$\gamma$ chain is a bi-infinite chain of trilobites along which the hierarchy
fails to compose; when infinitely many shifts are needed, the tiling itself is
not a substitution tiling: two half-planes, each the union of an increasing
chain of supertiles, meet along a bi-infinite line that carries a fragment occurring in no supertile. By contrast, our rules are
\emph{perfect} (Corollary~\ref{cor-perfect}): they admit exactly the
substitution tilings, and $\varphi$ carries this family \emph{onto} the entire
family of chair tilings (Theorem~\ref{th-chair}).

It seems worth recording that the Trilobite and Crab do not reach the
hierarchical tilings either, so that they appear to capture a class lying
strictly between the two. Indeed, by~\cite[p.~7]{gscrab} a tiling by these tiles
contains at most one $\gamma$ chain of $n$-trilobites for each $n$, any two such
chains are parallel, and the chain of level $n$ runs down the centre of the
chain of level $n+1$; the whole defect locus of such a tiling is therefore a
single straight line. Hierarchical chair tilings, however, also occur with a
bent defect: the increasing chain of supertiles around a tile may exhaust a
quadrant, rather than a half-plane or the whole plane, and the complementary
three quadrants are then filled by a hierarchy of their own, chosen
independently. The resulting tiling is hierarchical but not a substitution
tiling, and its defect runs along two perpendicular rays, so it cannot be the
image of a tiling by the Trilobite and Crab. We do not know whether the two
constructions are related beyond both being mapped onto the chair tilings by
local maps, and we leave this, together with a precise form of the above
observation, as a question for further work.

\section{Acknowledgments. }
The author is sincerely grateful to the participants of the Kolmogorov seminar and of the 
International academic conference “Graphs, Games and Models'' (October 12-15, 2022, Maikop, Adygea)
for helpful discussions.  Section~\ref{s4} arose thanks to the remark of an anonymous referee that
the triangle tilings from this paper are closely related  to chair tiling. The
author is grateful to the referee also for the counterexample to an earlier
version of Theorem~\ref{th-chair} and for pointing out that two chair tilings that differ
only along a line differ along a diagonal of supertiles.

\end{document}